\documentclass[a4paper]{article}

\usepackage{amsmath,amssymb,amsfonts,amsthm}

 \newtheorem{definition}{Definition}
 \newtheorem{theorem}[definition]{Theorem}
 \newtheorem{lemma}[definition]{Lemma}
 \newtheorem{corollary}[definition]{Corollary}
 \newtheorem{remark}[definition]{Remark}
 \newtheorem{example}[definition]{Example}
 \newtheorem{proposition}[definition]{Proposition}

\title{Transforms on operator monotone functions}
\author{Masato Kawasaki  and Masaru Nagisa}
\date{}

\begin{document}
\maketitle

\begin{abstract}
Let $f$ be an operator monotone function on $[0,\infty)$ with $f(t)\ge 0$ and $f(1)=1$.
If $f(t)$ is neither the constant function $1$ nor the identity function $t$, then
\[    h(t) = \frac{(t-a)(t-b)}{(f(t)-f(a))(f^\sharp (t)-f^\sharp (b))} 
     \qquad t\ge 0  \]
is also operator monotone on $[0,\infty)$, where $a,b\ge0$ and
\[  f^\sharp(t) = \frac{t}{f(t)} \qquad t\ge 0. \]
Moreover, we show some extensions of this statement.
\end{abstract}

\section{Introduction}
We call a real continuous function $f(t)$ on an interval $I$ operator monotone on $I$ (in short, $f\in \mathbb{P}(I)$ ),
if $A\le B$ implies $f(A)\le f(B)$ for any self-adjoint matrices $A, B$ with their spectrum containd in $I$.
In this paper, we consider only the case $I=[0,\infty)$ or $I=(0,\infty)$.
We denote $f\in \mathbb{P}_+(I)$ if $f\in \mathbb{P}(I)$ satisfies $f(t)\ge 0$ for any $t\in I$.

Let $\mathbb{H}_+$ be the upper half-plane of $\mathbb{C}$, that is,
\[  \mathbb{H}_+ = \{ z \in \mathbb{C} \mid {\rm Im }z > 0 \}
        = \{ z \in \mathbb{C} \mid |z|>0, \; 0<\arg z<\pi  \},  \]
where ${\rm Im} z$ (resp. $\arg z$) means the imaginary part (resp. the argument) of $z$.
When we choose an element $z\in \mathbb{H}_+$, we consider that its argument satisfies $0<\arg z<\pi$. 
As Loewner's theorem, it is known that $f$ is operator monotone on $I$ if and only if
$f$ has an analytic continuation to $\mathbb{H}_+$ that maps $\mathbb{H}_+$ into itself 
and also has an analytic continuation to the lower half-plane $\mathbb{H}_- (=-\mathbb{H}_+ )$,
obtained by the reflection across $I$
(see \cite{bhatia},\cite{hiai}).
For an operator monotone function $f(t)$ on $I$, we also denote by $f(z)$ its analytic continuation to $\mathbb{H}_+$.

D. Petz \cite{petz} proved that an operator monotone function 
$f:[0,\infty) \longrightarrow [0,\infty)$
satisfying the functional equation 
\[  f(t) = tf(t^{-1})  \qquad t\ge 0  \]
is related to a Morozova-Chentsov function 
which gives a monotone metric on the manifold of $n\times n$
density matrices.
In the work \cite{petzhase}, the concrete functions
\[  f_a(t) = a(1-a)\frac{(t-1)^2}{(t^a-1)(t^{1-a}-1)} \qquad (-1<a<2)  \]
appeared and their operator monotonicity was proved (see also \cite{caihan}).
V.E.S. Szabo introduced an interesting idea for checking their operator monotonicity in \cite{szabo}.
We use a similar idea as Szabo's in our argument.
M. Uchiyama \cite{uchiyama} proved the operator monotonicity of the following extended functions:
\[  \frac{(t-a)(t-b)}{(t^p-a^p)(t^{1-p}-b^{1-p})}    \]
for $0<p<1$ and $a,b>0$.
It is well known that the function $t^p$ $(0\le p\le 1)$ is operator monotone
as Loewner-Heinz inequality.
In this paper, we extend this statement to the following form:

\begin{theorem} \;  
Let $a$ and $b$ be non-negative real.
If $f\in \mathbb{P}_+[0,\infty)$ and both $f$ and $f^\sharp$ are not constant, 
then
\[  h(t) = \frac{(t-a)(t-b)}{(f(t)-f(a))(f^\sharp (t)-f^\sharp (b))}  \]
is operator monotone on $[0,\infty)$, where 
\[  f^\sharp(t) = \frac{t}{f(t)} \qquad t\ge 0 .  \]
\end{theorem}

We can also show the operator monotonicity of other functions which have the form related to
the above one in Theorem 4.

\section{Main result}

For $f\in \mathbb{P}[0,\infty)$, we have the following integral representation:
\[  f(z) = f(0) + \beta z + \int_0^\infty \frac{\lambda z}{z+\lambda}dw(\lambda), \]
where $\beta \ge 0$ and 
\[  \int_0^\infty \frac{\lambda}{1+\lambda}dw(\lambda)<\infty   \]
(see \cite{bhatia}).
When $f(0)\ge 0$ (i.e., $f\in \mathbb{P}_+[0,\infty))$, 
it holds that
$0<\arg f(z) \le  \arg z$ for $z\in \mathbb{H}_+$ (i.e., $0< \arg z<\pi$).

For any $f\in \mathbb{P}_+[0,\infty )$ ($f\neq 0$), we define $f^\sharp$ as follows:
\[  f^\sharp(t) = \frac{t}{f(t)} \qquad t\in[0,\infty).  \]
Then it is well-known that $f^\sharp \in \mathbb{P}_+[0,\infty)$.

\vspace{5mm}

\begin{proposition}
Let $f $ be an operator monotone function on $(0,\infty)$ and $a$ be positive real.
\begin{enumerate}
  \item[(1)] When $f(t)$ is not constant, the function 
\[  g_1(t) = \frac{t-a}{f(t)-f(a)}  .  \]
is operator monotone on $[0, \infty)$.
  \item[(2)] When $f(t)\ge 0$  for $t\ge 0$, the function
\[  g_2(t) = \frac{f(t)(t-a)}{tf(t)-af(a)}   \]
is operator monotone on $[0,\infty)$. 
\end{enumerate}
\end{proposition}
\begin{proof}\;
(1) It follows from Theorem 2.1 in \cite{uchiyama}.

(2) Since $f\in \mathbb{P}_+[0,\infty)$, we have $0<\arg zf(z) < 2\pi$ for any $z\in \mathbb{H}_+$.
So we can define
\[  g_2(z) =\frac{f(z)(z-a)}{zf(z)-af(a)},  \qquad  z\in \mathbb{H}_+  \]
and $g_2(z)$ is holomorphic on $\mathbb{H}_+$. 
Because $g_2([0,\infty))\subset [0,\infty)$ and $g_2(z)$ is continuous on
$\mathbb{H}_+ \cup[0,\infty)$, it suffices to show that $g_2(\mathbb{H}_+)\subset  \mathbb{H}_+$.
By the calculation
\begin{align*}
  g_2(z) & = \frac{zf(z)-af(a)+af(a)-f(z)a}{zf(z)-af(a)}
                    =1 -\frac{a(f(z)-f(a))}{zf(z)-af(a)}  \\
       & = 1 -\frac{a}{\dfrac{zf(z)-af(a)}{f(z)-f(a)} }
           = 1 -\frac{a}{z + f(a)g_1(z)},
\end{align*}
we have
\[  {\rm Im} g_2(z) = - {\rm Im}\frac{a}{z + f(a)g_1(z)} =  {\rm Im} \frac{a(z+f(a)g_1(z))}{|z+f(a)g_1(a)|^2}.  \]
When  $z\in \mathbb{H}_+$, ${\rm Im} g_1(z)>0$ by (1)
and  ${\rm Im}g_2(z)>0$.
So the function $g_2(t)$  belongs to $\mathbb{P}_+[0,\infty)$.
\end{proof}

\vspace{5mm}

For any  $z=e^{i\theta}$ ($0<\theta <\pi$) and any integer $n (\ge 2)$, we set
\[  w = \frac{\sin \theta}{\sin \dfrac{\pi+(n-1)\theta}{n}}e^{i(\pi+(n-1)\theta )/n}  .  \]
Since ${\rm Im}z = {\rm Im}w$,  $l = z - w >0$.
Then we can get
\[    \sup \{  l  \mid  0<\theta <\pi \}  = \lim_{\theta\to \pi -0}
      \frac{\sin \frac{\pi - \theta}{n}}{\sin \frac{(n-1)(\pi-\theta)}{n}} = \frac{1}{n-1}  .\]
So we have the following:

\begin{lemma} \;  For any $z\in \mathbb{H}_+$ and a positive integer $n$ ($n\ge 2$), we have
\[  \arg z < \arg (z-l) < \frac{\pi + (n-1)\arg  z}{n} \quad \text{if } \quad 0<l\le \frac{|z|}{n-1}.  \]
\end{lemma}

\vspace{5mm}

Now we can prove the following theorem and we remark that Theorem 1 easily follows from this:

\begin{theorem} \;  
Let $n$ be a positive integer,  $a,b,b_1, \ldots ,b_n \ge 0$ and $f, g, g_1,\ldots, g_n$ be non-constant, 
non-negative operator monotone functions on $[0,\infty)$.
\begin{enumerate}
  \item[(1)]  If $\dfrac{f(t)g(t)}{t}$ is operator monotone on $[0,\infty)$,
then the function
\[  h(t) = \frac{(t-a)(t-b)}{(f(t)-f(a))(g(t)-g(b))}   \]
is operator monotone on $[0.\infty)$ for any $a,b \ge 0$.
  \item[(2)]   If $\dfrac{f(t)}{\prod_{i=1}^n g_i(t)}$ is operator monotone on $[0,\infty)$,
then the function
\[  h(t) = \frac{(t-a)}{(f(t)-f(a))} \prod_{i=1}^n \frac{g_i(t)(t-b_i)}{tg_i(t)-b_ig_i(b_i)}   \]
is operator monotone on $[0.\infty)$ for any $a,b \ge 0$.
\end{enumerate}
\end{theorem}
\begin{proof}   \; (1)  By $f,g\in \mathbb{P}_+[0,\infty)$ and Proposition 2 (1),
\[  \frac{t-a}{f(t)-f(a)} \text{ and } \frac{t-b}{g(t)-g(b)}  \]
are operator monotone on $[0,\infty)$.
Therefore
\[  h(z) = \frac{(z-a)(z-b)}{(f(z)-f(a))(g(z)-g(b))}  \]
is holomorphic on $\mathbb{H}_+$, continuous on $\mathbb{H}_+\cup [0,\infty)$ 
and satisfies $h([0,\infty))\subset [0, \infty)$ and 
\[   \arg h(z) = \arg \frac{z-a}{f(z)-f(a)} + \arg \frac{z-b}{g(z)-g(b)}>0 
   \text{ for } z\in\mathbb{H}_+.  \]

We assume that $f(z)$ and $g(z)$ are continuous on  the closure $\overline{\mathbb{H}_+}$ of $\mathbb{H}_+$
and 
\[  f(t)-f(a)\neq 0 \text{ and } g(t)-g(b)\neq 0  \text{ for any } t\in(-\infty,0). \]
Then $h(z)$ is continuous on $\overline{\mathbb{H}_+}$.

In the case $z\in (-\infty,0)$, i.e.,  $|z|>0$ and $\arg z =\pi$,
we have
\begin{align*}
          &  \arg h(z)  \\
     =  &  \arg  (z-a) - \arg (f(z)-f(a)) + \arg (z-b) 
                                          - \arg (g(z)-g(b))  \\
      \le & \pi -\arg f(z)  + \pi - \arg g(z) \\
      \le &  2 \pi -\arg z  =\pi   \qquad (\text{since } \arg f(z) + \arg g(z) -\arg z \ge 0) .
\end{align*}
So it holds $0\le \arg h(z) \le \pi$.

In the case that $z\in \mathbb{H}_+$ satisifying $|z|>\max\{a,b\}$,
it holds that
\[  \arg (z-a),  \arg (z-b) < \frac{\pi + \arg z}{2}  \]
by Lemma 3.
Since
\begin{align*}
   \arg  h(z) & = \arg  (z-a) - \arg (f(z) - f(a)) + \arg (z-b) 
                                            - \arg  (g(z) - g(b))  \\
          & \le  \frac{\pi+ \arg z}{2} -\arg  f(z) + \frac{\pi+ \arg z}{2}
                                            - \arg  g (z) \\
          & = \pi +\arg z - \arg f(z) - \arg g(z) \le \pi,  
\end{align*}
we have $0 < \arg h(z) < \pi$.

For $r>0$, we define 
$H(r) = \{z\in \mathbb{C} \mid |z|\le r, {\rm Im}z \ge 0  \}$.
Whenever $r> l =\max \{a, b\}$, we can get
\[  0 \le \arg h(z) \le \pi  \]
on the boundary of $H(r)$.
Since $h(z)$ is holomorphic on $H(r)$,  ${\rm Im} h(z)$ is harmonic on $H(r)$.
Because ${\rm Im} h(z) \ge 0$ on the boundary of $H(r)$,
we have $h(H(r)) \subset \overline{\mathbb{H}_+}$ by the minimum principle of  harmonic functions.
This implies
\[  h( \overline{\mathbb{H}_+}) = h( \bigcup_{r>l}H(r))  \subset 
         \bigcup_{r>l} h(H(r))  \subset \overline{\mathbb{H}_+},   \]
and $h \in \mathbb{P}_+[0, \infty)$.

In general case,  
we set 
\[  \frac{f(t)g(t)}{t}=F(t) \text{ and } \tilde{f}(t)=\frac{f(t)}{F(t)}  \qquad (t\ge 0) .  \]
By the relation $\tilde{f}(t)g(t)=t$,  we have $\tilde{f} \in \mathbb{P}_+[0,\infty)$.
We define the function  $f_p$, $\tilde{f}_p$ and $g_p$  ($0<p<1$) as follows:
\[    f_p(z) = f(z^p) , \quad \tilde{f}_p(z)=\tilde{f}(z^p),  \]
and 
\[g_p(z) = (\tilde{f}_p)^\sharp(z)= \frac{z}{\tilde{f}_p(z)}=\frac{zF(z^p)}{f(z^p)}=z^{1-p} g(z^p) 
\]  
for  $z\in \overline{\mathbb{H}_+}$.  
Then we have $f_p$, $g_p \in \mathbb{P}_+[0,\infty)$ and
\[  h_p(z) = \frac{(z-a)(z-b)}{(f_p(z)-f_p(a))( g_p(z) - g_p(b) )}  \]
is holomorphic on $\mathbb{H}_+$ and continuous  on $\overline{\mathbb{H}_+}$.
By the fact  $\dfrac{f_p(t)g_p(t)}{t}=F(t^p)$  is operator monotone on $[0,\infty)$,
$h_p(t)$ becomes operator monotone on $[0,\infty)$.
Since 
\begin{align*}
    h_p(t) & = \frac{(t-a)(t-b)}
     {(f_p(t)-f_p(a))(g_p(t)-g_p(b))}  \\
           & = \frac{(t-a)(t-b)}
     {(f(t^p)-f(a^p))(t^{1-p}g(t^p)- b^{1-p}g(b^p))}  \qquad \text{for } t\ge 0 ,
\end{align*}
we have
\[    \lim_{p\to 1-0} h_p(t) = h(t)  .  \]
So we can get the operator monotonicity of $h(t)$.

(2)  We show this by the similar way as (1).
By Proposition 2,
\[ \frac{t-a}{f(t)-f(a)} \text{ and } \frac{g_i(t)(t-b_i)}{tg_i(t)-b_ig_i(b_i)}
  \quad (i=1,2,\ldots, n)  \]
are operator monotone on $[0,\infty)$.
So we have that 
\[  h(z) = \frac{z-a}{f(z)-f(a)}
   \prod_{i=1}^n \frac{g_i(z)(z-b_i)}{zg_i(z)-b_ig_i(b_i)} \]
is holomorphic on $\mathbb{H}_+$, 
continuous on $\mathbb{H}_+\cup [0,\infty)$ and satisfies
$h([0,\infty))\subset [0,\infty)$ and 
\[  \arg h(z)=\arg \frac{z-a}{f(z)-f(a)}+ \sum_{i=1}^n \arg \frac{g_i(z)(z-b_i)}{zg_i(z)-b_ig_i(b_i)}>0 \]
 for $z\in \mathbb{H}_+$.

We assume that $f(z)$ and $g_i(z)$ ($i=1,2,\ldots,n$) are continuous on 
$\overline{\mathbb{H}_+}$ and
\[  f(t)-f(a)\neq 0 \text{ and } tg_i(t)-b_ig_i(b_i)\neq 0 
   \text{ for any } t\in (-\infty,0).  \]
Then $h(z)$ is continuous on $\overline{\mathbb{H}_+}$.

In the case $z\in (-\infty,0)$, i.e.,  $|z|>0$ and $\arg z =\pi$,
we have
\begin{align*}
          &  \arg h(z)  \\
     =  &  \arg  (z-a) + \sum_{i=1}^n \arg g_i(z)(z-b_i) -\arg (f(z)-f(a))  
                                        - \sum_{i=1}^n \arg (zg_i(z)-b_ig_i(b_i))  \\
    \le & \pi + \sum_{i=1}^n \arg g_i(z) + n\pi -\arg f(z)  - n \pi 
        \quad (\text{since }\arg (zg_i(z)-b_ig_i(b_i)) \ge \pi) \\
    \le &  \pi   \qquad (\text{since } \arg f(z) - \sum_{i=1}^n \arg g_i(z) \ge 0) .
\end{align*}
So it holds $0\le \arg h(z) \le \pi$.

In the case $z\in \mathbb{H}_+$ satisifying $|z|>n\max\{a,b_1,b_2,\ldots,b_n\}$,
it holds that
\[  \arg (z-a),  \arg (z-b) < \frac{\pi + n\arg z}{n+1}  \]
by Lemma 3.
We may assume that there exists a number $k$ ($1\le k \le n$) such that
\[  \arg (zg_i(z))\le \pi \quad (i\le k), \qquad
    \arg (zg_i(z))>\pi \quad (i>k) . \]
Since
\begin{align*}
     & \arg  h(z) \\
   = & \arg  (z-a) + \sum_{i=1}^n \arg (z-b_i) +\sum_{i=1}^n \arg g_i(z) \\
     & \qquad  - \arg (f(z) - f(a)) - \sum_{i=1}^n \arg (zg_i(z) - b_ig_i(b_i))  \\
   \le & \frac{\pi+ n\arg z}{n+1}\times (n+1) + \sum_{i=1}^n \arg g_i(z) \\
       & \qquad -\arg  f(z) - \sum_{i=1}^k \arg zg_i(z) - (n-k)\pi \\
   = & \pi + n \arg z + \sum_{i=k+1}^n \arg g_i(z) - \arg f(z) - k\arg z -(n-k)\pi \\
   \le & \pi +(n-k) \arg z - (n-k) \pi \le \pi,  
\end{align*}
we have $0 \le \arg h(z) \le \pi$.

This means that  it holds
\[  0\le \arg h(z) \le \pi  \]
if $z$ belongs to the boundary of $H(r)=\{z\in\mathbb{C}\mid |z|\le r, {\rm Im} z\ge 0 \}$
for a sufficiently large $r$.
Using the same argument in (1), we can prove the operator monotonicity of $h$.

In general case, we define functions, for $p$ ($0<p<1$), as follows:
\[  f_p(t) = f(t^p), \quad g_{i,p}(t) = g_i(t^p) \quad (i=1,2,\ldots, n).  \]
Since $f, g_i\in \mathbb{P}_+[0,\infty)$,
\[ 0 < \arg f_p(z) <\pi, \quad 0 < \arg z g_{i,p}(z) <2 \pi  \]
for $z\in \mathbb{H}_+$.
This means that $f_p(z)$ and $g_{i,p}(z)$ are continuous on $\overline{\mathbb{H}_+}$ and 
\[ f_p(t)-f_p(a) \neq 0 \text{ and } tg_{i,p}(t)-b_ig_{i,p}(b_i)\neq 0 
   \text{ for any } t \in (-\infty, 0).  \]
Since
\[  \frac{f_p(t)}{\prod_{i=1}^n g_{i,p}(t)} 
  = \frac{f(t^p)}{\prod_{i=1}^n g_i(t^p)}  \qquad (0<p<1)\]
is operator monotone on $[0,\infty)$,
we can get the operator monotonicity of
\begin{align*}
   h_p(t) & = \frac{t-a}{f_p(t)-f_p(a)} 
   \prod_{i=1}^n \frac{g_{i,p}(t)(t-b_i)}{tg_{i,p}(t)-b_ig_{i,p}(b_i)} \\
   & = \frac{t-a}{f(t^p)-f(a^p)}
   \prod_{i=1}^n \frac{g_i(t^p)(t-b_i)}{tg_i(t^p)- b_ig_i(b_i^p)}.
\end{align*} 
So we can see that
\[  h(t) = \lim_{p\to 1-0} h_p(t)  \]
is operator monotone on $[0,\infty)$.
\end{proof}

\vspace{5mm}

\begin{remark}
Using Proposition 2 and Theorem 4, we can prove the operator monotonicity of 
the concrete functions in \cite{petzhase}.
Since $t^a$ ($0<a<1$) and $\log t$ is operator monotone on $(0,\infty)$,
\begin{align*}
  f_a(t) & = a(1-a)\frac{(t-1)^2}{(t^a-1)(t^{1-a}-1)} \qquad (-1<a<2) \\
    & = 
    \begin{cases} 
       a(a-1)\dfrac{t^{-a}(t-1)^2}{(t^{-a}-1)(t\cdot t^{-a}-1)}  \quad & -1<a<0 \\
       \dfrac{t-1}{\log t}  & a=0,1 \\
       a(1-a)\dfrac{(t-1)^2}{(t^a-1)(t^{1-a}-1)} & 0<a<1 \\
       a(a-1)\dfrac{t^{a-1}(t-1)^2}{(t^{a-1}-1)(t\cdot t^{a-1}-1)} & 1<a<2
    \end{cases}
\end{align*}
becomes operator monotone.
\end{remark}

\vspace{5mm}

\begin{corollary} Let $f\in \mathbb{P}_+(0,\infty)$ and both $f$ and $f^\sharp$
be not constant.
For any $a>0$, we define
\[  h_a(t) = \frac{(t-a)(t-a^{-1})}{(f(t)-f(a))(f^\sharp(t)-f^\sharp(a^{-1}))}
   \qquad t\in (0,\infty).  \]
Then we have
\begin{enumerate}
  \item[(1)] $h_a$ is operator monotone on $(0,\infty)$.
  \item[(2)] $f(t)=t\cdot f(t^{-1})$ implies $h_a(t)=t\cdot h_a(t^{-1})$.
  \item[(3)] $a=1$ and $f(t^{-1})=f(t)^{-1}$ implies 
$h_1(t)=t \cdot h_1(t^{-1})$.
\end{enumerate}
\end{corollary}
\begin{proof} \;
We can directly prove (1) from theorem 3.
Because
\begin{align*}
       t\cdot h_a(t^{-1})  
  &  =  \frac{t(t^{-1}-a)(t^{-1}-a^{-1})}{(f(t^{-1})-f(a))(f^\sharp(t^{-1}) - f^\sharp(a^{-1}))}  \\
  & =  \frac{(t-a)(t-a^{-1})}{t(f(t^{-1})-f(a))(f^\sharp(t^{-1})-f^\sharp(a^{-1}))},
\end{align*}
we can compute
\begin{align*}
    & t(f(t^{-1})-f(a))(f^\sharp(t^{-1})-f^\sharp(a^{-1})) - (f(t)-f(a))(f^\sharp(t)-f^\sharp(a^{-1}))  \\
  = & (f(t^{-1})-f(a))(1/f(t^{-1})-t/af(a^{-1})) - (f(t) - f(a))(t/f(t)- 1/af(a^{-1}))  \\
  = & 0
\end{align*}
if it holds $f(t)=t\cdot f(t^{-1})$ or $a=1$, $f(t^{-1})=f(t)^{-1}$.
So we have (2) and (3).
\end{proof}

\begin{example}
Using this corollary, we can repeatedly construct an operator monotone function $h(t)$ on $[0,\infty)$
satisfying the relation
\[  h(t)=t\cdot h(t^{-1})  \qquad t>0 \tag{*} .  \]
If we choose $t^p$ ($0<p<1$) as $f(t)$ in  Corollary 5(3),
\[   h(t)  = \frac{(t-1)^2}{(t^p-1)(t^{1-p}-1)}.  \]
If we choose $\dfrac{(t-1)^2}{(t^p-1)(t^{1-p}-1)}$ as $f(t)$ in Corollary 5(2),
\begin{align*}  h(t) = & \frac{t-a} {\dfrac{(t-1)^2}{(t^p-1)(t^{1-p}-1)}- 
        \dfrac{(a-1)^2}{(a^p-1)(a^{1-p}-1) } } \\
       &  \qquad \times   
        \frac{t-a^{-1}}{ \dfrac{t(t^p-1)(t^{1-p}-1)}{(t-1)^2)}-\dfrac{a(a^{-p}-1)(a^{p-1}-1)}{(a-1)^2}  }  
\end{align*}
for $a>0$.
If we choose $t^p+t^{1-p}$  ($0<p<1$) as $f(t)$ in Corollary 5(2),
\begin{align*}
   h(t)  & =  \frac{t-a}{t^p+t^{1-p}-a^p-a^{1-p}}  \times
           \frac{t-a^{-1}}{ \dfrac{1}{t^{p-1}+t^{-p}} - \dfrac{1}{a^p+a^{1-p}} }
           \qquad (a>0)  \\
    & = \frac{\sqrt{t}(\cosh(\log t)-\cosh(\log a) )}
           {\cosh(\log \sqrt{t})-\cosh(\log \sqrt{t} +\log(t^p+t^{1-p})-\log(a^p+a^{1-p}) ) }.
\end{align*}
These functions, $h\in \mathbb{P}_+[0,\infty)$, satisfy the relation (*).
\end{example}

\section{Extension of Theorem 4}

Let  $m$ and $n$ be  positive integers and $f_1, f_2,\ldots, f_m$, $g_1,g_2,\ldots,g_n $ be non-constant,
non-negative operator monotone functions on $[0,\infty)$.
We assume that the function
\[   F(t) = \frac{\prod_{i=1}^m f_i(t)}{t^{m-1}\prod_{j=1}^n g_j(t)}   \]
is operator monotone on $[0,\infty)$.
For non-negative numbers $a_1, a_2,\ldots, a_m, b_1,b_2,\ldots,b_n$, we define the function $h(t)$ as follows:
\[  h(t) = \prod_{i=1}^m \frac{t-a_i}{f_i(t)-f_i(a_i)} \prod_{j=1}^n \frac{g_j(t)(t-b_j)}{tg_j(t)-b_jg_j(b_j)}
     \qquad (t\ge 0)  . \]
Then it follows from Proposition 2 that $h(z)$ is holomorphic on $\mathbb{H}_+$ , 
$h([0,\infty))\subset [0,\infty)$ and $\arg h(z) > 0$ for any $z \in \mathbb{H}_+$.

\begin{theorem}  In tha above setting, we have the followings: 

(1)  When $f_i$ and $g_j$ ($1\le i \le m$, $i\le j \le n$) are continuous on $\overline{\mathbb{H}_+}$ and 
\[  f_i(t) - f_i(a_i) \neq 0, \quad  t g_j(t)-b_jg_j(b_j) \neq 0, \quad t\in (-\infty,0),  \]
$h(t)$ is operator monotone on $[0,\infty)$.
 
(2)  When there exists a positive number $\alpha$ such that $\alpha \arg z \le \arg F(z)$ for all $z\in \mathbb{H}_+$,
$h(t)$ is operator monotone on $[0,\infty)$.
\end{theorem}
\begin{proof}   \; 
(1) Using the same argument of proof of Theorem 4 (1), it suffices to show that
$0\le \arg h(z) \le \pi$ for $z\in \mathbb{R}$ or  $z\in \mathbb{H}_+$ whose absolutely value is sufficiently large.

In the case $z\in (-\infty,0)$, i.e.,  $|z|>0$ and $\arg z =\pi$,
we have
\begin{align*}
     & \arg h(z) \\
   = & \sum_{i=1}^m \arg (z-a_i) + \sum_{j=1}^n \arg (g_j(z)(z-b_j) ) \\
     & \qquad - \sum_{i=1}^n \arg (f_i(z) - f_i(a_i)) 
              -\sum_{j=1}^n (zg_j(z)-b_jg_j(b_j) )  \\
 \le & m\pi + n\pi + \sum_{j=1}^n \arg g_j(z) 
         - \sum_{i=1}^m \arg f_i(z) - n\pi  \\
   = & \pi -\arg \frac{\prod_{i=1}^n f_i(z)}{z^{m-1}\prod_{j=1}^n g_j(z)}
 \le \pi .
\end{align*}
So it holds $0\le \arg h(z) \le \pi$.

In the case that $z\in \mathbb{H}_+$ satisifies 
\[  |z|> (m+n-1) \max \{ a_i, b_j \mid 1 \le i \le m, 1\le j \le n \}. \]
Then it holds that
\[  \arg (z-a_i), \arg (z-b_j)  < \frac{\pi + (m+n-1)\arg z}{m+n}  \]
by Lemma 3.
We may assume that there exists $k$ ($1\le k \le n$) such that
\[  \arg (z g_j(z)) \le \pi \quad (j\le k), \qquad
    \arg (z g_j(z)) > \pi \quad (j>k) .  \]
Since
\begin{align*}
     & \arg  h(z) \\
 \le & \frac{\pi +(m+n-1)\arg z}{m+n}\times m +\frac{\pi +(m+n-1)\arg z}{m+n}\times n 
              + \sum_{j=1}^n \arg g_j(z) \\
     & \qquad -\sum_{i=1}^m \arg f_i(z) - \sum_{j=1}^k \arg zg_j(z) -(n-k)\pi \\
  =  & \pi + (m+n-k-1)\arg z + \sum_{j=k+1}^n \arg g_j(z) 
        - \sum_{i=1}^m \arg f_i(z) -(n-k)\pi  \\
 \le & \pi + (n-k)(\arg z - \pi) 
        - \arg \frac{\prod_{i=1}^m f_i(z)}{z^{m-1}\prod_{j=1}^m g_j(z)}  \\
 \le & \pi -\arg F(z) \le \pi ,  
\end{align*}
we have $0 \le \arg h(z) \le \pi$.
So $h(t)$ is operator monotone on $[0,\infty)$.

(2) We choose a positive number $p$ as follows:
\[  \frac{m-1}{\alpha +m-1} < p < 1  .  \]
We define functions $f_{i,p}, g_{j,p}$ as follows:
\[  f_{i,p}(z) = f_i(z^p), \quad g_{j,p}(z) = g_j(z^p) 
    \qquad (z\in \mathbb{H}_+). \]
Since $f_i, g_j \in \mathbb{P}_+[0,\infty)$, $f_{i,p}, g_{j,p}$ are continuous on $\overline{\mathbb{H}_+}$ and
satisfy the condition
\[  f_{i,p}(t) -f_{i,p}(a_i) \neq 0, \quad tg_{j,p}(t)-b_jg_{j,p}(b_j)\neq 0, \quad t\in(-\infty,0). \] 
We put
\[  F_p(t) = \frac{\prod_{i=1}^m f_{i,p}(t)}{t^{m-1}\prod_{j=1}^n g_{j,p}(t)}
           = F(t^p) t^{-(m-1)(1-p)} . \]
Then $F_p$ is holomorphic on $\mathbb{H}_+$ and satisfies $F_p( (0,\infty) )\subset (0,\infty)$.
For any $z\in \mathbb{H}_+$, we have
\begin{align*}
  \arg F_p(z) & = \arg F(z^p) - (m-1)(1-p)\arg z \le \arg F(z^p) \le \pi  \\ 
\intertext{and}  
  \arg F_p(z) & \ge \alpha \arg z^p -(m-1)(1-p)\arg z  \\ 
                    & =(\alpha p - (m-1)(1-p))\arg z \\
                    & =( (\alpha +m -1)p-(m-1) ) \arg z >0 . 
\end{align*}
So we can see $F_p\in \mathbb{P}_+[0,\infty)$.
By (1), we can show that
\[  h_p(t) = \prod_{i=1}^m \frac{(t-a_i)}{f_{i,p}(t)-f_{i,p}(a_i)} 
                 \prod_{j=1}^n \frac{g_{j,p}(t)(t-b_j)}{tg_{j,p}(t) - b_jg_{j,p}(b_j)}  \]
is operator monotone on $[0,\infty)$.
When $p$ tends to $1$, $h_p(t)$ also tends to $h(t)$.
Hence $h(t)$ is operator monotone on $[0,\infty)$.
\end{proof}

\begin{example}
Let $0<p_i\le1$ ($i=1,2,\ldots,m$) and $0\le q_j \le 1$ ($j=1,2,\ldots,n$).
We put
\[  f_i(t) = t^{p_i}, \quad g_j(t) = t^{q_j} \qquad (t\ge0).  \]
By the calculation
\[  F(t) = \frac{\prod_{i=1}^m f_i(t)}{t^{m-1}\prod_{j=1}^n g_j(t)} = t^{\sum_{i=1}^m p_i -\sum_{j=1}^n q_j -(m-1)}, \]
we have, for real numbers $a_i$, $b_j\ge 0$, 
\[  h(t) = t^{\sum_{j=1}^n q_j} \frac{(t-a_1)\cdots (t-a_m)(t-b_1)\cdots (t-b_n)}
                  {(t^{p_1}-a_1^{p_1})\cdots (t^{p_m}-a_m^{p_m})(t^{1+q_1}-b_1^{1+q_1})\cdots (t^{1+q_n}-b_n^{1+q_n}) }  \]
is operator monotone on $[0,\infty)$ if it holds
\[  0\le \sum_{i=1}^m p_i -\sum_{j=1}^n q_j -(m-1) \le 1,  \]
i.e., $F(t)$ is operator monotone on $[0,\infty)$.

When $\sum_{i=1}^m p_i = \sum_{j=1}^n q_j +(m-1)$, we can see that
\[  h(t) = \frac{t^{\sum_{j=1}^m q_j}  (t-1)^{m+n} }{\prod_{i=1}^m (t^{p_i}-1) \prod_{j=1}^n (t^{1+q_j} -1) }  \]
is operator monotone on $[0,\infty)$ and satisfies the functional equation
\[  h(t) = t \cdot h(t^{-1}) . \tag{*} \]

We can easily check that, if $h_1$, $h_2\in \mathbb{P}_+[0,\infty)$ satisfy the property (*), then the functions
\begin{align*}
  f(t) & =h_1(t)^{1/p} h_2(t)^{1-1/p}  \qquad (0<p<1) \\
  g(t) & = \frac{t}{h_1(t)}
\end{align*}
are operator monotone on $[0,\infty)$ and satisfy the property (*).

Combining these facts, for $r_i, s_i$ ($i=1,2,\ldots,n$) with
\begin{gather*}
 0<r_1, \ldots ,r_c \le 1, \quad 1\le r_{c+1}, \ldots , r_n \le 2  \\
 0<s_1, \ldots ,s_d \le 1 , \quad 1\le s_{d+1}, \ldots , s_n \le 2 \\
\sum_{i=1}^c r_i = \sum_{i=c+1}^n r_i -1, \quad \sum_{i=1}^d s_i = \sum_{i=d+1}^n s_j -1 ,
\end{gather*}
we can see that the function
\[  h(t) = \sqrt{ t^\gamma  
              \prod_{i=1}^n \frac{r_i (t^{s_i}-1)}{s_i (t^{r_i}-1)} }  \]
is operator monotone on $[0,\infty)$ and satisfies the property (*) and $h(1)=1$,
where $\gamma = 1-c+d+\sum_{i=1}^c r_i - \sum_{i=1}^d s_i$.
\end{example}

\vspace{5mm}
\noindent
{\bf Acknowledgements.} \; The authors thank to Professor M. Uchiyama for his 
useful comments. 
The works of M. N. was partially supported by Grant-in-Aid for Scientific Research
(C)22540220.

\noindent
Graduate School of Science\\
Chiba University\\
Inage-ku, Chiba 263-8522\\
Japan\\
e-mail: nagisa@math.s.chiba-u.ac.jp\\

\end{document}